\documentclass[11pt]{amsart}
\usepackage{amsmath,amsfonts, mathtools,amssymb,latexsym,amssymb,amsthm,adjustbox,mathrsfs,amsbsy, bm,tikz-cd,indentfirst, makecell,graphicx, verbatim, calc, enumerate,xy}
\usepackage[top=1in,bottom=1.3in,left=1in,right=1in]{geometry}
\usepackage[colorlinks=true,linkcolor=blue]{hyperref}
\usepackage{listings}
\usepackage{xcolor}
\theoremstyle{plain}
\newtheorem{theorem}{Theorem}[section]
\newtheorem{lemma}[theorem]{Lemma}

\newtheorem{corollary}[theorem]{Corollary}
\newtheorem{proposition}[theorem]{Proposition}
\newtheorem{remark}[theorem]{Remark}
\newtheorem{definition}[theorem]{Definition}

\newtheorem{question}[theorem]{Question}

\newtheorem{example}[theorem]{Example}
\newtheorem{point}[theorem]{}

\newcommand{\m}{\mathfrak{m}}

\newcommand{\wt}{\widetilde}

\newcommand{\grade}{\operatorname{grade}}
\newcommand{\depth}{\operatorname{depth}}

\newcommand{\Hom}{\operatorname{Hom}}

\newcommand{\Ext}{\operatorname{Ext}}

\newcommand{\type}{\operatorname{type}}
\usepackage{graphicx}
\usepackage{subfigure}
\usepackage{epsfig}
\usepackage{epstopdf}
\usepackage{float} 
\theoremstyle{plain}

\usepackage{xcolor}
\usepackage{titlesec}

\titleformat{name=\section}{}{\thetitle.}{0.5em}{\centering\scshape\large}

\lstdefinelanguage{Macaulay2}{morekeywords={ideal,ring,quotient,vars,degree,dim,syz,res,betti,hilbertSeries,gb,monomialIdeal,primaryDecomposition,radical,Ext,Tor,Hom,Spec,Proj,Sheaf,Module,ChainComplex,Map,apply,for,do,if,then,else,end, cokernel, gens, reduceHilbert},sensitive=true,morecomment=[l]{--},morestring=[b]"}

\allowdisplaybreaks

\begin{document}
\title{Integrally closed ideals with $e_{2}(I)=e_{1}(I)-e_{0}(I)+\lambda(A/I)$}
\author{Shruti Priya}
\address{ Indian Institute of Technology Kharagpur, Kharagpur 721302, West Bengal, India} \email{shruti96312@kgpian.iitkgp.ac.in}
\author{Samarendra Sahoo}
\address{ Indian Institute of Technology Dharwad, Dharwad 580011, Karnataka, India}
\email{samarendra.s.math@gmail.com}


\begin{abstract} 
Let $(A,\mathfrak{m})$ be a Cohen-Macaulay local ring of dimension $d.$ We introduce and study the notion of the generalized type of $A$ with respect to an $\mathfrak{m}$-primary ideal $I$ denoted by $\operatorname{type}_I(A).$ Let $e_i(I)$ denote $i$th Hilbert coefficients of $A$ w.r.t. $I$. Assuming $I$ is integrally closed and $e_{2}(I)=e_{1}(I)-e_{0}(I)+\lambda(A/I) \neq 0$, we establish a sharp lower bound for $\operatorname{type}_I(A)$ in terms of the multiplicity and certain lengths associated to $I.$  We further show that when this lower bound is attained, the associated graded ring  $G(I)$, is Cohen-Macaulay.  In the case of Buchsbaum local rings of dimension $d$ and depth at least $d-1$, we obtain an optimal lower bound for $e_{2}(\m)$ using the technique of $S_{2}$-fication.  Additionally, for an integrally closed $\mathfrak{m}$-primary ideal $I,$ we  also study the second extremal case  $e_{2}(I)=e_{1}(I)-e_{0}(I)+\lambda(A/I)+1$ and its  consequences on $G(I).$ We also investigate bounds on $e_3(I)$ and for $d=3,$ we  study the consequences when these bounds are attained for.

\end{abstract}

\maketitle

\section{introduction} 
The study of Hilbert coefficients has long been a central theme in commutative algebra and algebraic geometry. These coefficients encode subtle information about the singularities of a ring and capture important algebraic properties of the blowup algebras associated to an ideal $I$, namely the Rees algebra of $I$ denoted by $\mathcal{R}(I)=\displaystyle \bigoplus_{n \geq 0}I^{n}t^{n}\subseteq A[t]$, where $t$ is an indeterminate over $A$, and 
the associated graded ring of $I$ denoted by $G(I)=\displaystyle \bigoplus_{n \geq 0}I^{n}/I^{n+1}$. In particular, the Hilbert coefficients play a crucial role in understanding the depth and Cohen-Macaulayness of these graded structures. We
refer the interested readers to the book \cite{rossi2010hilbert} for the 
main developments and rich references in this direction.

To state the problems addressed and the results obtained in this paper, we first fix some terminologies. Let $(A,\mathfrak{m})$ be a Noetherian local ring of dimension $d$ and $I$ be an $\m$-primary ideal. 
 The  \textit{Hilbert-Samuel function} of $I$ is defined as $H_{I}(n):= \lambda(A/I^{n+1})$ (where $\lambda($\textunderscore$)$ denotes the length).
     For sufficiently large  $n,$ this function coincides with a polynomial $P_{I}(x) \in \mathbb Q[x]$ of degree $d,$  called the \textit{Hilbert-Samuel polynomial} of $I$.  This polynomial can be expressed using binomial coefficients as:
\begin{equation*}
    P_{I}(n)=e_{0}(I)\dbinom{n+d}{d}-e_{1}(I)\dbinom{n+d-1}{d-1}+\cdots+(-1)^{d}e_{d}(I).
\end{equation*}
The unique integers $e_{0}(I), \ldots, e_{d}(I)$ are referred to as the \textit{Hilbert  coefficients} of $A$ with respect to $I$. The leading coefficient $e_{0}(I)$ of the polynomial $P_{I}(n)$ is called the \textit{multiplicity} of $A$ with respect to $I.$ For Cohen-Macaulay rings, $e_{0}(I)$ and $e_{1}(I)$ are non-negative integers, and  a celebrated result of Narita \cite{Narita_1963} shows that  $e_{2}(I)$ is a non-negative integer. Beyond the case where the associated graded ring has good properties on its depth, relatively little is known about the higher Hilbert coefficients.

 Given an ideal $I$, the integral closure of $I$, denoted by  $\overline{I}$ is characterized as the largest ideal containing $I$, such that  $e_{0}(I)=e_{0}(\overline{I})$. Ratliff-Rush introduced the ideal $\widetilde{I}=\displaystyle \bigcup_{n\geq 1}\left(I^{n+1}:I^n\right),$ known as the \textit{Ratliff-Rush closure} of $I,$ which is the largest ideal containing $I$ with the same Hilbert coefficients as $I.$ In particular, if $I$ is a regular ideal then $I \subseteq \widetilde{I} \subseteq \overline{I}.$ Hence, if $I$ is integrally closed, then  $\widetilde{I}=I,$ and we say that $I$ is \textit{Ratliff-Rush closed}. Itoh \cite{itoh1} established that for an integrally closed $\m$-primary ideal $I,$  $e_{2}(I)\geq e_{1}(I)-e_{0}(I)+\lambda(A/I),$ and showed that $G(I)$ is Cohen-Macaulay whenever $e_2(I)\in \{0,1,2\}.$  Motivated by this, Valla \cite[Conjecture 6.20]{valla1998problems} conjectured that the equality for $I=\m$ implies the Cohen-Macaulayness of $G(\m)$; however,  Wang \cite[3.10]{corso2005depth} showed that this fails in dimension $d \geq 2.$ Therefore, the equality alone does not imply that $G(\m)$ is Cohen-Macaulay. However, Corso, Polini, and Rossi \cite{corso2005depth} proved that the conjecture holds under the stronger assumption that $I$ is normal, i.e., $I^{n}$ is integrally closed for every $n \geq 1$. Subsequently, Mishra and Puthenpurakal \cite{ankit} provided sufficient conditions for the conjecture to hold in the case $I=\mathfrak{m}$. More precisely, they proved that if $e_{2}(\mathfrak{m}) = e_{1}(\mathfrak{m}) - e_{0}(\mathfrak{m}) + 1,$ then the Cohen-Macaulay type  satisfies $\type(A) \geq e_{0}(\mathfrak{m}) - h - 1,$ where $h = \mu(\mathfrak{m}) - d$. Furthermore, they proved that Valla’s conjecture holds whenever the Cohen-Macaulay type attains this lower bound.

 In this paper,  motivated by these developments, we investigate the depth of $G(I)$ under the assumption $e_{2}(I)=e_{1}(I)-e_{0}(I)+\lambda(A/I) \neq 0,1,2$ with $I$ integrally closed.  We show that when dimension $d \leq 1,$ this condition forces  $G(I)$ to be Cohen-Macaulay; whereas in dimension $2$ the $\depth G(I)$ exhibits extremal behavior (see Lemma \ref{lemma4.4}). This leads us to identify additional conditions that ensure $G(I)$ is Cohen-Macaulay. 
To identify such conditions, we examine the behavior of the Ratliff-Rush filtration of $I$ (see \ref{ratliffrushfiltration}), which plays a key role in controlling the depth of the associated graded ring. Recall that the Ratliff-Rush filtration of $I$  
is said to \textit{behave well modulo a superficial element} $x$ for $I$, if $\widetilde{I^{n}}A_{1}=\widetilde{I^{n}A_{1}}$ for all $n \geq 1,$ where $A_{1}=A/(x).$ In \cite{Part2}, Puthenpurakal introduced this notion and established several equivalent criteria for such a behavior. In particular,  in \cite[Theorem~6.2]{Part2}, he proved that if $e_2(I) = \cdots = e_d(I) = 0,$ then the associated graded ring with respect to the Ratliff-Rush filtration, $\widetilde{G}(I)=\displaystyle \bigoplus_{n \geq 0} \widetilde{I^{n}}/\widetilde{I^{n+1}}$, is Cohen-Macaulay. In particular, the Ratliff-Rush filtration behaves well modulo a superficial sequence $x_1, \ldots, x_{d-1} \in I$ (see \cite[Definition 4.4]{Part2} for the definition). In contrast, we show that, for $d=3$, under the numerical condition $e_{2}(I)=e_{1}(I)-e_{0}(I)+\lambda(A/I) \neq 0$ together with $e_{3}(I)=0,$ the vanishing of a single higher Hilbert coefficient suffices to guarantee the good behavior of the Ratliff-Rush filtration. Therefore, we prove the following result:

 \begin{proposition}
     
 [Proposition \ref{e3bd}] \label{e3cm}
 Let $(A,\m) $ be a Cohen-Macaulay local ring of dimension $d \geq 2$ and let $I$ be an $\m$-primary integrally closed ideal. If $e_2(I)=e_1(I)-e_0(I)+\lambda(A/I)\neq 0$, then the following hold:
\begin{enumerate} [\normalfont(i)]
     \item  For $d=2,$ $e_3(I)\leq 0.$ Further, if  $e_3(I)=0$ then $G(I)$ is Cohen-Macaulay.

    \item For $d\geq 3,$ $e_3(I)\leq 0.$ Further, if $d=3$ and $e_3(I)=0$, then $\wt{G}(I)$ is Cohen-Macaulay. 
\end{enumerate}
\end{proposition} 
Note that in \cite[Theorem 2.8]{mafi}, the authors proved that under the hypothesis of Proposition \ref{e3cm}, if $d=3$, then $e_3(I)\leq 0.$

Inspired by the role of the Cohen-Macaulay type,  we introduce the notion of \textit{generalized type} of $A$ with respect to an ideal $I$ as follows: $$\type_I(A)=\lambda_A\left(\Ext^t_{A}(A/I,A)\right),$$ 
where $\depth A=t.$  We prove that $\mathrm{type}_{I}(A)$ is preserved modulo a non-zero divisor of $A$ (see Proposition \ref{type}). Further, we establish a lower bound on this invariant and prove that: 
  \begin{theorem}[Theorem \ref{bd}]\label{theorem2}
   Let $(A,\m)$ be a Cohen-Macaulay local ring of dimension $d$ and let $I$ be an $\m$-primary integrally closed ideal. Let $J$ be a minimal reduction of $I.$ If $e_{2}(I)=e_{1}(I)-e_{0}(I)+\lambda(A/I) \neq 0$ then $\type_I(A) \geq e_{0}(I)-\lambda(A/I)-\lambda(I/I^{2}+J)\geq 0.$
\end{theorem}

As a direct consequence of Theorem \ref{theorem2}, we show that equality implies that  $G(I)$ is Cohen-Macaulay (see Theorem \ref{typeq}). In this case, we also describe the associated 
 $h$-polynomial  of $I$.

 Furthermore, a natural generalization of Cohen-Macaulay rings is Buchsbaum rings.  Recall that a Noetherian local ring $(A, \mathfrak m)$ is \textit{Buchsbaum} if every system of parameter $x_{1},\ldots,x_{d}$ of $A$ is a weak sequence, i.e., 
     $(x_{1},\ldots,x_{i-1}):x_{i}=(x_{1},\ldots,x_{i-1}):\mathfrak m \text{ for } i=1,\ldots, d.$
It is therefore natural to ask whether Itoh’s bounds \cite{itoh1} on $e_{2}(I),$ for integrally closed ideals $I,$ continue to hold for Buchsbaum rings. Motivated by this question, we prove the following result:

 \begin{proposition}[Proposition \ref{buchsbaumring}] 
    Let $(A,\m)$ be a Buchsbaum local ring of dimension $d \geq 2$ with $\depth A\geq d-1$, then $e_{2}(\m)\geq e_{1}(\m)-e_{0}(\m)+1.$
\end{proposition}

Our second main study concerns the next extremal case
$e_{2}(I)=e_{1}(I)-e_{0}(I)+\lambda(R/I)+1$ with $I$ integrally closed. When $\dim A \leq 1,$ we explicitly describe the $h$-polynomial of $I,$ and prove that $\deg h_{I}(t)=3.$ In dimension two, we prove that if $I^{3} \nsubseteq J$ for any minimal reduction $J$ of $I$, then $\widetilde{G}(I)$ is Cohen-Macaulay if and only if $\wt{I^2}\cap J=JI$ (see Proposition \ref{depth}). We also show that if $d \ge 2$, then $e_3(I) \le 1$. Thus, we prove the following result:

\begin{theorem}[Theorem \ref{e3=1}]
    Let $(A, \m)$ be a Cohen-Macaulay local ring of dimension $d\geq2$ and let $I$ be an integrally closed $\m$-primary ideal of $A$. Let $J$ be a minimal reduction of $I$. Let $e_{2}(I)=e_{1}(I)-e_{0}(I)+\lambda(A/I)+1$ and $I^3\nsubseteq J$. If $\wt{I^2}\cap J=JI$, then the following hold:
    \begin{enumerate}[\normalfont(i)]
          \item If $d=2$ then $e_3(I)\leq 1$ and $G(I)$ is Cohen-Macaulay if equality holds.
        \item If $d\geq 3$ then $e_3(I)\leq 1$ and $\wt{G}(I)$ is Cohen-Macaulay if equality holds and $d=3$.
    \end{enumerate}
\end{theorem}

\textbf{Organization of the paper.} This paper is organized into four sections. Section 2 presents some necessary preliminary results. In Section 3, we prove Proposition 1.1 and Theorem 1.2 and study their consequences on the depth of the associated graded ring. This section also contains the proof of Proposition 1.3. Section 4 is devoted to the second extremal case, where we prove Theorem 1.4. We provide examples in support of our claims.

\section{preliminaries}

In this section, we recall the basic definitions and preliminary results that will be used throughout the paper and fix our notations and conventions. Unless stated otherwise, $(A,\mathfrak m)$ denotes a $d$-dimensional Noetherian local ring with an infinite residue field $k=A/\m$, and $I$ denotes an $\mathfrak m$-primary ideal of $A$. For undefined terms, we redirect readers to \cite{Bruns}.

\begin{point} \normalfont
    The \textit{Hilbert series} $HS_{I}(t)$ of $I$,  is the formal power series $\displaystyle \sum_{n \geq 0} \lambda(I^{n}/I^{n+1}) t^{n}$. 
    By Hilbert-Serre theorem, we write 
    $ HS_{I}(t)=\dfrac{h_{I}(t)}{(1-t)^{d}}$,
where $h_{I}(t) \in \mathbb{Z}[t]$ is the unique polynomial with $h_{I}(1) \neq 0$, known as the \textit{h-polynomial} of $I.$   The power series, $\displaystyle \sum_{n \geq 0} H_{I}(n)t^{n}$ is called the \textit{Hilbert-Samuel series} of $I.$ Note that $\displaystyle \sum_{n \geq 0} H_{I}(n)t^{n}=\frac{h_{I}(t)}{(1-t)^{d+1}},$
and an easy computation shows that for all $i \geq 0,$ we have $e_{i}(I)=\dfrac{h^{(i)}_{I}(1)}{i!},$
 where $h^{(i)}_{I}(1)$ is the $i^{th}$-formal derivative of $h_{I}(t)$ at $t=1.$
\end{point}

\begin{point} \label{ratliffrushfiltration}
\normalfont
The \textit{Ratliff-Rush filtration} with respect to $I$, is the filtration  $\mathcal{F}=\{\widetilde{I^{n}}\}_{n \geq 0},$ where $\wt{I^n}=\displaystyle \bigcup_{k\geq 1}(I^{n+k}:I^k).$ It is well known that if $\grade(I,A)>0$, then  $\widetilde{I^{n}}=I^{n}$ for sufficiently large $n$ (see \cite{MR506202}).  Set $s^{*}(I)=\min\{k \in \mathbb{N}: \widetilde{I^{n}}=I^{n} \text{ for all } n \geq k\}.$
This numerical invariant is known as the \textit{stability index} of the Ratliff-Rush filtration or the \textit{Ratliff-Rush index.}
\end{point}

\begin{point}\label{superf}
 \normalfont
    An element $x\in I $ is called a \textit{superficial element} for $I,$ if there exist a positive integer $c$ such that $(I^{n+1}:x) \cap I^{c}=I^{n}$ for all $n \geq c.$ If $\depth(I, A) > 0$, then one can show that an A-superficial element is A-regular. Furthermore,
in this case 
 $(I^{n+1} :x) = I^n$ for sufficiently large $n.$ Superficial elements exist if the residue field is infinite (see \cite[Pg. 7]{sally1978numbers}). 
 
 A sequence $x_{1},x_{2},\ldots,x_{j}$ is  a \textit{superficial sequence} for $I,$ if $x_{1}$ is a superficial element for $I$, and $x_{i}$ is a superficial element for $I/(x_{1},x_{2},\ldots,x_{i-1})$ for all $i=2,\ldots,j.$ 

 If $\dim A\geq 2$ and $I$ is integrally closed, then a result of Itoh \cite[Page 648]{itoh1} guarantees the existence of a superficial element $x \in I$ such that the ideal $I/(x)$ is an integrally closed ideal.

\end{point}

\begin{point}\label{reduction}
    \normalfont   An ideal $J \subseteq I$ is a \textit{reduction} of $I$ if $JI^{n}=I^{n+1}$ for some $n \in \mathbb{N,}$ 
and  \textit{minimal reduction} if it is minimal with respect to inclusion among all reductions. The \textit{reduction number} of $I$ with respect to $J$ is defined as 
   $r_{J}(I):=\min\{n: JI^{n}=I^{n+1}\}.$ From \cite[page 12]{rossi2010hilbert}, if $I$ is $\m$-primary and the residue field is infinite, then the ideal generated  by any maximal superficial sequence for $I$ is a minimal reduction of $I$.
\end{point}

\begin{point}\cite[Definition 15 and Theorem 16]{HCCMM} 
\normalfont
Let $(A,\m)$ be a Cohen-Macaulay local ring of dimension $d\geq 1$ and $I$ an $\m$-primary ideal. Then $e_1(I)\geq e_0(I)-\lambda(A/I).$ Moreover, equality holds if and only if  $G(I)$  has \textit{minimal multiplicity}, i.e., $\deg h_{I}(t) \leq 1$; in this case, $G(I)$  is Cohen-Macaulay. 
\end{point}

\section[]{ $e_{2}(I)= e_{1}(I)-e_{0}(I)+\lambda(A/I)$}
   In this section, we study the integrally closed $\m$-primary ideal $I$ in a Cohen-Macaulay ring $(A, \m)$ satisfying  the numerical condition $e_{2}(I)= e_{1}(I)-e_{0}(I)+\lambda(A/I),$ and its consequences  on $\depth G(I).$
We begin by proving that for $d \leq 1$, 
the degree of the 
$h$-polynomial of $I$ is at most two and that the associated graded ring  $G(I)$ is Cohen-Macaulay. In dimension two, the associated graded ring exhibits a dichotomous behavior: its depth is either $0$ or $2$. Further, in this section, we introduce notion of  the generalized type of a ring $A$ with respect to an ideal $I,$ and establish a lower bound for it. Consequently, we prove that if this lower bound is achieved, then the associated graded ring is Cohen-Macaulay.
Note that if  $e_2(I)=e_{1}(I)-e_{0}(I)+\lambda(A/I)=0$ then $e_{1}(I)=e_{0}(I)-\lambda(A/I)$, and hence,  by \cite[Theorem 16]{HCCMM}, $G(I)$ is Cohen-Macaulay. Therefore, throughout this section, we assume  $e_{2}(I)= e_{1}(I)-e_{0}(I)+\lambda(A/I) \neq 0.$

\begin{proposition} \label{prop4.1}
    Let $(A, \m) $ be a Cohen-Macaulay local ring of dimension less than or equal to $1,$ and $I$ an integrally closed $\m$-primary ideal. If $e_{2}(I)=e_{1}(I)-e_{0}(I)+\lambda(A/I)$, then 
    \begin{enumerate}[\normalfont(i)]
        \item \label{prop4.1(i)} $\deg h_{I}(t)\leq 2$,
        \item \label{prop4.1(ii)} $G(I)$ is Cohen-Macaulay.
    \end{enumerate}
\end{proposition}
\begin{proof}
(\ref{prop4.1(i)})
 Let $d=0$ and $\deg h_{I}(t)=s.$ Set $\rho_{j}(I)=\lambda(I^{j}/I^{j+1})$ for all $j\geq 0,$ then  $h_{I}(t)=\displaystyle \sum_{j=0}^{s}\rho_{j}(I)t^{j}.$  Therefore, the Hilbert coefficients are given by 
        $$ e_{0}(I)=\sum_{j=0}^{s}\rho_{j}(I), \hspace{10pt} e_{1}(I)=\sum_{j=1}^{s}j\rho_{j}(I), \hspace{10pt} e_{2}(I)=\sum_{j=2}^{s}\binom{j}{2}\rho_{j}(I).$$
Since $e_{2}(I)=e_{1}(I)-e_{0}(I)+\lambda(A/I),$   we get $\rho_{3}(I)+3\rho_{4}(I)+\ldots+\dbinom{s-1}{2} \rho_{s}(I)=0.$ This implies that $\rho_{j}(I)=0$ for all $j \geq 3.$ Hence, $\deg h_I(t) \leq 2,$ and therefore, $h(t)=\rho_{0}(I)+\rho_{1}(I)t+\rho_{2}(I)t^2.$

Suppose $d=1$. Set $\nu_{j}(I)=\lambda(I^{j+1}/JI^{j})$ for all $j\geq 0,$ where $J=(x)$ is a minimal reduction of $I.$ From \cite[Theorem 2.5$(c)$]{rossi2010hilbert}, we have $e_{1}(I)=\displaystyle \sum_{j\geq0}\nu_{j}(I)$ and $e_{2}(I)=\displaystyle \sum_{j\geq1}j\nu_{j}(I).$ Hence, we get
\begin{equation*}
    \begin{split}
          e_{2}(I)&
    =\displaystyle \sum_{j\geq2}j\nu_{j}(I)+\nu_{1}(I)-\sum_{j\geq2}\nu_{j}(I)-\nu_{0}(I)-\nu_{1}(I)+e_{1}(I)\\
    &=\displaystyle \sum_{j\geq2}(j-1)\nu_{j}(I)+e_{1}(I)-e_{0}(I)+\lambda(A/I).
    \end{split}
\end{equation*}
  Since $e_{2}(I)=e_{1}(I)-e_{0}(I)+\lambda(A/I)$, we get $ \displaystyle \sum_{j\geq2}(j-1)\nu_{j}(I)=0$, this implies $\nu_{j}(I)=0$ for all $j \geq 2.$ From \cite[Theorem 2.5$(c)$]{rossi2010hilbert}, we get $\deg h_I(t) \leq 2,$ and thus,   $h_{I}(t)=\nu_{0}(I)+(e_{0}(I)-\nu_{1}(I)-\lambda(A/I))t+(\nu_{1}(I)-\nu_{2}(I))t^{2}.$
  
\noindent
(\ref{prop4.1(ii)}) For $d=0,$ nothing to prove. Suppose $d=1.$ 
Let $J=(x)$ be a minimal reduction of $I$, where $x \in I\setminus I^2$ is a superficial element for $I$. From (\ref{prop4.1(i)}), we have $\nu_{j}(I)=0$ for all $j\geq 2$. Consequently, $I^3=JI^2,$ and therefore, $r_{J}(I)\leq 2.$ Since $I$ is integrally closed, we have $\widetilde{I}=I.$ Moreover, from \cite[Proposition 4.2$(i)$]{MR2013172}, it is known that $s^{*}(I) \leq r_{J}(I),$ hence $s^{*}(I)=1.$ Therefore, $G(I)$ is Cohen-Macaulay. 
\end{proof}

 In Proposition \ref{prop4.1}, to prove (\ref{prop4.1(ii)}), it suffices to assume that $I$ is Ratliff-Rush closed. Since throughout this paper we work with integrally closed ideals, we make this assumption here as well.
The following example shows that $I$ being Ratliff-Rush closed is necessary.
\begin{example} \cite[Example 1.4.2]{rossi2010hilbert} \normalfont Let $A=\mathbb{Q}[[t^4,t^5,t^6,t^7]]$ and  $I=(t^4,t^5,t^6).$ Note that  $t^7\in (I^2:I)\setminus I$. Therefore, $I\neq \wt{I}.$  This implies that $I$ is not Ratliff-Rush closed, and hence not integrally closed. 
By CoCoA \cite{cocoa}, the Hilbert Series for $I$ is given by:
\begin{equation*}
    HS_{I}(t)=\dfrac{2+t+t^2}{(1-t)}.
\end{equation*}
Hence, $e_{0}(I)=4,$ $e_{1}(I)=3$ and $e_{2}(I)=1.$ Here,  $e_{2}(I)=e_{2}(I)-e_{2}(I)+\lambda(A/I).$ But, since $s^{*}(I)>1$, thus $G(I)$ is not Cohen-Macaulay. 
\end{example}

\begin{lemma} \label{lemma4.4}
    Let $(A,\m)$ be a two-dimensional Cohen-Macaulay local ring and let $I$ be an $\m$-primary integrally closed ideal. If $e_{2}(I)=e_{1}(I)-e_{0}(I)+\lambda(A/I)$  then $\depth G(I)=0$ or $2$.
\end{lemma}

\begin{proof}
     Since the residue field of $A$ is infinite, we  choose $x \in I\setminus I^2$ a superficial element for $I,$ such that $I/(x)$ is integrally closed. Set $A_1=A/(x)$ and $I_1=IA_1.$ If $\depth G(I)=0,$ there is nothing to prove. Hence, assume that  $\operatorname{depth}G(I) \neq 0.$ Since $\lambda(A_1/I_1)=\lambda(A/I),$ and  by \cite[Proposition 1.2]{rossi2010hilbert}, we have $e_{i}(I_1)=e_{i}(I)$ for $0\leq i\leq2.$ Thus, $e_{2}(I_1)=e_{1}(I_1)-e_{0}(I_1)+\lambda(A_1/I_1)$.  Therefore, by Proposition \ref{prop4.1} (\ref{prop4.1(ii)}), it follows that $\operatorname{depth}G(I_1)=1.$ Hence, by Sally's descent $\operatorname{depth}G(I)=2.$
\end{proof}

We now give examples illustrating that both possible values of $\depth G(I)$ described in Lemma \ref{lemma4.4} can occur.

\begin{example} \normalfont
(i)  \cite[Example 3.10]{corso2005depth}Let $A=\mathbb{K}[[x,y,t,u,v]]/(t^2,tu,tv,uv,yt-u^3,xt-v^3),$ with $\mathbb{K}$ a field and $x,y,t,u,v$ indeterminates. Let $\m$ be the maximal ideal.  By CoCoA \cite{cocoa}, the Hilbert Series of $\m$ is:
        \begin{equation*}
            HS_{\m}(t)=\frac{1+3t+3t^2-t^4}{(1-t)^2}.
        \end{equation*}
        Here $e_{0}(\m)=6$, $e_{1}(\m)=8$ and $e_{2}(\m)=3.$ In particular,  $e_{2}(\m)=e_{1}(\m)-e_{0}(\m)+\lambda(A/\m)$ and  $\depth G(\m)=0.$ 

        \noindent (ii) Let $A=\mathbb Q[[X,Y,Z]]/(X^3+Y^5+Z^5),$  where $X,Y,Z$ are variables and   let $x,y,z$ denote the images of $X,Y,Z$ in $A.$ Let $I= \overline{(x^2,z^4)}=(x^2,xz^2,xyz,xy^2,z^4,yz^3,y^2z^2,y^3z,y^4).$ By CoCoA \cite{cocoa}, the Hilbert Series of $I$ is:
        \begin{equation*}
            HS_{I}(t)=\frac{13+24t+3t^2}{(1-t)^2}.
        \end{equation*}
       Here $e_{0}(I)=40$, $e_{1}(I)=30$ and $e_{2}(I)=3.$ In particular,  $e_{2}(I)=e_{1}(I)-e_{0}(I)+\lambda(A/I)$ and $ \depth G(I)=2.$
\end{example}

The behavior of higher Hilbert coefficients is, in general, subtle and difficult to control, even for integrally closed $\m$-primary ideals in Cohen–Macaulay local rings. While $e_{0}(I), e_{1}(I)$ and $e_{2}(I)$ are always non-negative, much less is known about the sign and vanishing of the higher coefficients. Motivated by a result of Mafi and Naderi \cite[Theorem 2,8]{mafi},  we establish the following theorem. It shows that the third Hilbert coefficient $e_3(I)$ is always non-positive under the given hypothesis that $e_{2}(I)=e_{1}(I)-e_{0}(I)+\lambda(A/I)$. Moreover, the vanishing of $e_3(I)$ leads to several desirable properties, especially the Cohen–Macaulayness of $G(I)$ or $\wt{G}(I)$, and good behavior of the Ratliff–Rush filtration modulo a superficial sequence of length $d-1$.

\begin{proposition}\label{e3bd}
Let $(A,\m) $ be a Cohen-Macaulay local ring of dimension $d \geq 2$ and let $I$ be an $\m$-primary integrally closed ideal. If $e_2(I)=e_1(I)-e_0(I)+\lambda(A/I)\neq 0$, then the following hold:
\begin{enumerate} [\normalfont(i)]
    \item \label{e3bd(i)} For $d=2,$ $e_3(I)\leq 0.$ Further, if  $e_3(I)=0$ then $G(I)$ is Cohen-Macaulay.

    \item \label{e3bd(ii)} For $d\geq 3,$ $e_3(I)\leq 0.$ Further, if $d=3$ and $e_3(I)=0$ then $\wt{G}(I)$ is Cohen-Macaulay. 
\end{enumerate}
\end{proposition}

\begin{proof}

   (\ref{e3bd(i)}) Let $J$ be a minimal reduction of $I.$ Set $\wt{\nu_j}(I)=\lambda(\wt{I^{j+1}}/J\wt{I^j})$ for all $j\geq 0.$ From \cite[Equation 3.1]{rossi2010hilbert}, we have $e_1(I)=\sum_{j\geq 0}\wt{\nu_j}(I)$ and $e_2(I)=\sum_{j\geq 1}j\wt{\nu_j}(I).$ From the hypothesis, we obtain $\wt{\nu_j}(I)=0$ for all $j\geq 2.$ Thus, from \cite[Theorem 2.5$(c)$]{rossi2010hilbert}, we have
\begin{equation}\label{al}
    \wt{e}_{3}(I)=\displaystyle \sum _{j \geq 2}\binom{j}{2}\wt{\nu_j}(I)=0.
\end{equation}
 Furthermore, from \cite[1.5$(b)$]{Part2}, we also have
\begin{equation} \label{e3}
    e_{3}(I)=\wt{e}_{3}(I)-\sum_{j\geq 0}\lambda\left(\frac{\widetilde{I^{j+1}}}{I^{j+1}}\right)=-\sum_{j\geq 0}\lambda\left(\frac{\widetilde{I^{j+1}}}{I^{j+1}}\right)\leq 0.
\end{equation}
    Suppose that equality holds in (\ref{e3}), then $\widetilde{I^{j}}=I^{j}$ for all $j \geq 1.$ This implies $\depth G(I) >0.$ Then the result follows from Lemma \ref{lemma4.4}.

\noindent
(\ref{e3bd(ii)}) We proceed by induction on the dimension $d$. Suppose $d=3.$ We choose $x \in I\setminus I^2$ a superficial element for $I,$ such that $I/(x)$ is integrally closed. Set ${A_1}=A/(x)$ and ${I_1}=I/(x)$. From \cite[Proposition 1.2]{rossi2010hilbert}, we have $e_i(I)=e_i(I_1)$ for $i=0,1,2$ and
   \begin{align} \nonumber \label{123}
       e_{3}(I)&=e_{3}({I_1})+\sum_{j\geq 0}\lambda\left((I^{j+1}:x)/I^{j}\right)\\
       &=\widetilde{e}_{3}({I_1})-\sum_{j\geq 0}\lambda\left(\widetilde{{I^{j+1}_1}}/{I^{j+1}_1}\right)+\sum_{j\geq 0}\lambda\left((I^{j+1}:x)/I^{j}\right) \text{ (from \cite[1.5$(b)$]{Part2})}.
   \end{align}
    
 For all $j\geq 0$, we have the following exact sequence: 
\begin{equation}\label{behave}
    0\longrightarrow (I^{j+1}:x)/I^j\longrightarrow  \wt{I^j}/I^j \longrightarrow  \wt{I^{j+1}}/I^{j+1} \xlongrightarrow{g_{j+1}} \wt{{I^{j+1}_1}}/{I^{j+1}_1} \longrightarrow \mathrm{Coker}(g_{j+1}) \longrightarrow 0
\end{equation}
Taking lengths and summing over all $j \geq 0$,   we get:
\begin{equation*}
    \sum_{j\geq 0}\lambda\left((I^{j+1}:x)/I^j\right)+\sum_{j\geq 0}\lambda(\wt{I^{j+1}}/I^{j+1})-\sum_{j\geq 0}\lambda(\wt{I^j}/I^j)-\sum_{j\geq 0}\lambda\left(\wt{{I^{j+1}_1}}/{I^{j+1}_1}\right) \leq 0.
\end{equation*}
This implies that $\displaystyle\sum_{j\geq 0}\lambda\left((I^{j+1}:x)/I^j\right)-\sum_{j\geq 0}\lambda\left(\wt{{I^{j+1}_1}}/{I^{j+1}_1}\right)\leq 0.$ Note that the inequality is strict if the map $g_{j}$ is not surjective for some $j\geq 1.$
 Now, from equations \eqref{al} and (\ref{123}), we obtain 
    \begin{equation} \label{e30}
        e_3(I)=\widetilde{e}_{3}({I_1})-\sum_{j\geq 0}\lambda\left(\widetilde{{I^{j+1}_1}}/{I^{j+1}_1}\right)+\sum_{j\geq 0}\lambda\left((I^{j+1}:x)/I^{j}\right)\leq \widetilde{e}_{3}({I_1})=0.
    \end{equation}
    
    Suppose equality holds in (\ref{e30}), then the map $g_j$ is surjective for all $j\geq 1.$ Therefore, from \cite[Definition 4.4]{Part2}, the Ratliff-Rush filtration of $I$ behaves well modulo $(x),$ this implies $\depth \widetilde{G}(I) \geq 2.$ Furthermore, from \cite[Remark 4.5]{Part2}, $\widetilde{G}(I)/x\widetilde{G}(I)=\widetilde{G}({I_1}).$ Since $\widetilde{G}(I_1)$ is Cohen-Macaulay by \cite[Proposition 3.4$(iv)$]{Part2}, it follows from Sally's descent that $\widetilde{G}(I)$ is Cohen-Macaulay. 

    Assume that the statement is true for $d-1.$ Suppose $d\geq 4.$ Let $x \in I\backslash I^{2}$ be a superficial element for $I$ such that $I/(x)$ is integrally closed.
    From \cite[Proposition 1.2]{rossi2010hilbert}, $e_{3}(I)=e_{3}(I_1)$. By induction hypothesis, $e_3(I)\leq 0.$ 
\end{proof}

We now discuss an example showing that for $d\geq 4$, if $e_3(I)=0$, then $\wt{G}(I)$ may not be Cohen-Macaulay. This example is the same as Example 5.2(1) in \cite{ankit}.

\begin{example}
\normalfont
     Let $A=k[[x,y,t,z,u,v,w,s_1,\ldots , s_{d-3}]]/
    (z^2,zu,zv,uw,zw,uv,vw,yz-u^3,xz-v^3,tz-w^3)$ and $I=\m.$ Note that $A$ is Cohen-Macaulay local ring of dimension $d$ and
    $\operatorname{depth} G(\m)=d-3$. It follows that $G(\m)=\wt{G}(\m)$ and $\wt{G}(\m)$ is not Cohen-Macaulay. We also have
    $$ h_I(t)=1+4t+6t^3-4t^4+t^5.$$ It can be checked that $e_2(\m)=e_1(\m)-e_0(\m)+1=4\neq 0$ and $e_3(\m)=0$.  
\end{example}

We now define the generalized type of $A$ with respect to an ideal.

\begin{definition}\label{def} \normalfont
    Let $(A,\m)$ be a local ring of depth $t$ and $I$ an ideal of $A$ (not necessarily $\m$-primary). The \textit{generalized type} of $A$ with respect to $I$ is denoted as $\type_I(A)$ and defined as: $$\type_I(A)=\lambda_A\left(\Ext^t_{A}(A/I,A)\right).$$ 
\end{definition}
Note that if $I$ is $\m$-primary, $\type_I(A)<\infty$ and for $I=\m$, this invariant coincides with the Cohen-Macaulay type $A$. In the following result, we prove that  $\type_I(A)$ is preserved modulo a non-zero divisor  of $A.$
\begin{proposition}\label{type}
     Let $(A,\m)$ be a local ring and $I$ an ideal of $A$ with $\depth(I, A)=t>0$. Let $x\in I$ be a  non-zero divisor of $A.$ Then $\type_I(A)=\type_{I/(x)}(A/(x)).$
\end{proposition}

\begin{proof}
    Since $x\in I$ is a non-zero divisor of $A$, depth$(A/(x))=t-1$. From Definition \ref{def}, $$\type_{I/(x)}(A/(x))=\lambda_{A/(x)}\left(\Ext^{t-1}_{A/(x)}(A/I,A/(x))\right)\cong \lambda_{A}\left(\Ext^{t}_{A}(A/I,A)\right)=\type_I(A).$$ The isomorphism is due to \cite[Lemma 3.1.16]{Bruns}.
\end{proof}

In the following theorem, we establish a lower bound for $\type_{I}(A).$

\begin{theorem}\label{bd}
   Let $(A,\m)$ be a Cohen-Macaulay local ring of dimension $d$ and let $I$ be an $\m$-primary integrally closed ideal. Let $J$ be a minimal reduction of $I.$ If $e_{2}(I)=e_{1}(I)-e_{0}(I)+\lambda(A/I) \neq 0$ then $$\type_I(A) \geq e_{0}(I)-\lambda(A/I)-\lambda(I/I^{2}+J)\geq 0.$$
\end{theorem}

\begin{proof}

    We proceed by induction of the dimension $d.$ Let $d=0.$ Then $J=0.$ By  Proposition \ref{prop4.1}, we have $\deg h_{I}(t) \leq 2.$ Consequently, $I^{3}=0.$ This implies that $\displaystyle I^{2} \subseteq (0:_{A}I)\cong \Hom_{A}(A/I,A).$ If $I^{2}=0,$ then  $\deg h_{I}(t) \leq 1,$ and hence, $e_{2}(I)=0,$ which is a contradiction. Thus, $I^{2}\neq 0,$ and as $\lambda(I^2)=e_{0}(I)-\lambda(A/I)-\lambda(I/I^2),$ we get $\type_I(A) \geq e_{0}(I)-\lambda(A/I)-\lambda(I/I^{2})=e_{0}(I)-\lambda(A/I)-\lambda(I/I^{2}+J).$

    Suppose $d=1.$ From Proposition~\ref{prop4.1}, $G(I)$ is Cohen-Macaulay, therefore, we get 
    \begin{align*} 
        \type_I(A)&=\type_{I_1}(A_1) & \text{ (from Proposition \ref{type})}\\
        &\geq e_{0}(I_{1})-\lambda(A_1/I_{1})-\lambda(I_{1}/I_{1}^{2}) & \text{ (from $d=0$ case)}\\
        &=e_{0}(I)-\lambda(A/I)-\lambda(I/I^{2}+J) & \text{  (from \cite[Proposition 1.2]{rossi2010hilbert})}.
    \end{align*}

    Let $d \geq 2.$  We choose $x_{1},\ldots,x_{d} \in I\setminus I^2$  a superficial sequence for $I$ such that $J=(x_{1},\ldots,x_{d})$ is a minimal reduction of $I$. For $s=1,\ldots ,d$,  set $A_s=A/(x_{1},\ldots,x_{s})$ and $I_s=IA_{s}$. Note that  $I_s$ is integrally closed for all $s=1,\ldots , d-1.$

    Suppose $d=2.$  
Set $\wt{\nu_j}(I)=\lambda(\wt{I^{j+1}}/J\wt{I^j})$ for all $j\geq 0.$ By   \cite[Equation 3.1]{rossi2010hilbert}, we have $ \displaystyle e_2(I)=\sum_{j\geq 1}j\wt{\nu_j}(I)$ and $\displaystyle e_1(I)=\sum_{j\geq 0}\wt{\nu_j}(I).$ Since $e_{2}(I)=e_{1}(I)-e_{0}(I)+\lambda(A/I)$, we obtain $\wt{\nu_j}(I)=0$ for all $j\geq 2$. Therefore, $I^3\subseteq\wt{I^3}=J\wt{I^2}.$ Going module $J$, we get $I^3_d=0$ in $A_d.$ Furthermore, if $I^2_d=0$, then $I^2\subseteq J.$ Since $I$ is integrally closed, we have $I^2\cap J=JI$. It follows that $I^2=JI$. Therefore, from \cite[Theorem 16]{HCCMM}, $e_{0}(I)-\lambda(A/I)=e_{1}(I),$ this implies $e_{2}(I)=0,$ which is a contradiction. Thus, $I^2_d\neq0$. Now since $I^2_d \subseteq (0:_{A_{d}}I_{d})$ and $\type_I(A)=\type_{I_d}(A_d)$, therefore
 $\type_I(A) \geq e_{0}(I)-\lambda(A/I)-\lambda(I/I^{2}+J).$

    Suppose $d\geq 3.$ From Proposition \ref{type} and \cite[Proposition 1.2]{rossi2010hilbert}, we have $\type_I(A)=\type_{I_{d-2}}(A_{d-2})$ and $e_i(I)=e_i(I_{d-2})$ for $i=0,1,2$, respectively. Moreover, $\lambda(A/I)$ and $\lambda(I/I^2+J)$ remain unchanged upon passing modulo a superficial sequence. Thus, by $d=2$ case, $\type_I(A) \geq e_{0}(I)-\lambda(A/I)-\lambda(I/I^{2}+J).$

Since $e_0(I)=\lambda(A/J)$, we have $\type_I(A)\geq \lambda(I/J)-\lambda(I/I^{2}+J)\geq 0.$ This proves the last inequality. 
\end{proof}

The following remark is stated in \cite[3.1]{ankit} for the maximal ideals. However, the same argument works for any Ratliff-Rush closed ideal. We therefore state the following in this generality.

\begin{remark} \label{imprmk} \normalfont 
    Let $(A,\m)$ be a local ring, and let $I$ be a Ratliff-Rush-closed ideal. Let $x\in I$ be a superficial element for $I$. Set $A_{1}=A/(x)$ and $I_1=I/(x).$ Then we have the following:
    \begin{enumerate}[{(i)}]
        \item If $\dim A_{1}>0,$ then  we have the following inclusion: $0\longrightarrow \wt{I^2}/I^2\longrightarrow \wt{I_1^2}/I_1^2.$
        \item If $\dim A_{1}=0,$ then  we have the following inclusion map: $\psi:\wt{I^2}/I^2\longrightarrow {I_1}/I_1^2.$
    \end{enumerate}
\end{remark}
In the next result, we discuss the consequence when $\type_I(A) = e_{0}(I)-\lambda(A/I)-\lambda(I/I^{2}+J)$ in Theorem \ref{bd}.  

\begin{theorem}
    \label{typeq}
    (with the same hypothesis as in Theorem \ref{bd}) If $\type_I(A) = e_{0}(I)-\lambda(A/I)-\lambda(I/I^{2}+J)$ then $G(I)$ is Cohen-Macaulay.
\end{theorem}

\begin{proof}
    If $d\leq 1$ then by Proposition \ref{prop4.1}, $G(I)$ is Cohen-Macaulay.
Assume $d \geq 2.$ We choose $x_{1},\ldots,x_{d} \in I\setminus I^2$  a superficial sequence for $I$ such that $J=(x_{1},\ldots,x_{d})$ is a minimal reduction of $I$. For $s=1,\ldots ,d$,  set $A_s=A/(x_{1},\ldots,x_{s})$ and $I_s=IA_{s}$. Note that  $I_s$ is integrally closed for all $s=1,\ldots , d-1.$

Suppose $d=2.$  From \cite[Proposition 3.4]{Part2}, we have $\wt{G}(I)$ is Cohen-Macaulay and ${\wt{I^j}}A_{1}=\wt{I^j_1}$ 
for all $j\geq 1.$ From the proof of Theorem \ref{bd}, we get $\wt{I^{j+1}}=J\wt{I^j}$ for all $j\geq 2.$ This implies that $I^3_2=0$ in $A_2.$  Since, $\wt{I^3}=J\wt{I^2}$, we get $ {I\wt{I^2}}A_{1}\subseteq {\wt{I^3}}A_{1}={J\wt{I^2}}A_{1} ={J}{\wt{I^2_1}}A_{1} \subseteq (x_2)A_1.$
This implies that $I_1\wt{I^2_1}\subseteq \wt{I^3_1}\subseteq (x_2)A_1.$
Going modulo $x_2$, we get $I_2{\wt{I^2_1}}A_{2}=0$ which implies $ I^2_2\subseteq {\wt{I^2_1}}A_{2}\subseteq (0:_{A_2}I_2).$

Thus, $ e_{0}(I)-\lambda(A/I)-\lambda(I/I^{2}+J)=\lambda(I^2_2)\leq \type_{I_2}(A_2).$ By the hypothesis, we have $I^2_2={\wt{I^2_1}}A_{2}.$ 
 From  Remark \ref{imprmk} (ii) and (i), we obtain $I_{1}^{2}=\widetilde{I_{1}^{2}}$ and  $I^{2}=\widetilde{I^{2}}$ respectively. Consequently,
$$I^{3}\subseteq \widetilde{I^{3}}=J\widetilde{I^{2}}=JI^{2}\subseteq I^{3},$$
and hence $I^{3}=\widetilde{I^{3}}$. Iterating this process, we conclude that
$I^{j}=\widetilde{I^{j}} \quad \text{for all } j\geq 2.$
Since $I$ is an integrally closed ideal, it follows that $\depth G(I)>0$.
Therefore, by Lemma~\ref{lemma4.4}, the associated graded ring $G(I)$ is
Cohen-Macaulay.

  Suppose $d\geq 3.$ From Proposition \ref{type} and \cite[Proposition 1.2]{rossi2010hilbert}, we have $\type_I(A)=\type_{I_{d-2}}(A_{d-2})$ and $e_i(I)=e_i(I_{d-2})$ for $i=0,1,2$. Moreover, $\lambda(A/I)$ and $\lambda(I/I^2+J)$ do not change upon passing modulo a superficial sequence. Therefore, the result follows from the case $d=2$ and Sally's descent. 
\end{proof}

\begin{remark}  \normalfont 
\begin{enumerate} [(i)]
    \item \textit{(with the same hypothesis as in Theorem \ref{bd})}
   Theorems \ref{bd} and \ref{typeq} are generalizations of \cite[Proposition~2.21 and Theorem~3.1]{ankit} to the class of integrally closed ideals. Indeed, when $I=\mathfrak m$, we have
$\lambda(A/\mathfrak m)=1$ and $$\lambda(\mathfrak m/\mathfrak m^2+J)
= \lambda(\mathfrak m/\mathfrak m^2)
- \lambda((\mathfrak m^2+J)/\mathfrak m^2)
= \mu(\mathfrak m)-d
= h.$$

\item  \textit{(with the same hypothesis as in Theorem \ref{typeq})}  If $e_{2}(I)=e_{1}(I)-e_{0}(I)+\lambda(A/I) \neq 0$ and $\type_I(A) = e_{0}(I)-\lambda(A/I)-\lambda(I/I^{2}+J)$ then $$h_I(t)=\lambda(A/I)+\lambda(I/I^2+J)t+(e_0(I)-\lambda(A/I)-\lambda(I/I^2+J))t^2.$$
\end{enumerate}
\end{remark}

The following example illustrates Theorem \ref{typeq}. The authors gratefully acknowledge Prof. Clare D'Cruz and Dr. Sudeshna Roy for their assistance with the computations for this example using Macaulay 2 \cite{macaulay2} and CoCoA \cite{cocoa}.

\begin{example} \normalfont
   Let $A=\mathbb Q[[X,Y,Z]]/(X^4+Y^7+Z^7),$ where $X,Y,Z$ are variables and   let $x,y,z$ denote the images of $X,Y,Z$ in $A.$ Let $I= \overline{(y^2,z^2)}=(x^2,y^2,z^2,xy,yz,zx)$. Note that  $A$ is a Cohen-Macaulay ring of dimension 2 and $\depth G(I)=2.$ By CoCoA \cite{cocoa}, the Hilbert series of $I$ is:
   \[HS_{I}(t)=\frac{4 + 8t + 4t^2}{(1-t)^{2}}.\]

   Hence, $e_{0}(I)=16$, $e_{1}(I)=16$ and $e_{2}(I)=4.$ In particular, $e_{2}(I)=e_{1}(I)-e_{0}(I)+\lambda(A/I).$ Here $J=(y^2,z^2)$ is a minimal reduction of $I$ as $JI^2=I^3$ and $\lambda(A/J)=16=e_{0}(I)$ 
with   $\lambda(I/I^2+J)=8.$  Note that $\type_{I}(A)= e_{0}(I)-\lambda(A/I)-\lambda(I/I^2+J)=4.$ The computation of $\type_{I}(A)$
 was carried out using the following Macaulay 2 \cite{macaulay2} commands.  
\begin{lstlisting}
i1 : R = QQ[x,y,z];
i2 : K = ideal(x^4+y^7+z^7);
i3 : S = R/K;
i4 : I = ideal(x^2,y^2,z^2,x*y,y*z,z*x);
i5 : M = cokernel gens I;
i6 : N = Ext^2(M,S);
i7 : degree N
o7 = 4
\end{lstlisting}
\end{example}

We conclude this section by extending Itoh’s result \cite{itoh1} to the maximal ideal of Buchsbaum local rings with depth at least $d-1$, obtaining a similar lower bound for the second Hilbert coefficient. Before, we briefly recall the construction of the  $S_{2}$-ification of a Buchsbaum ring from the literature. 

\begin{point} \normalfont \cite[Section 4]{GHEZZI_GOTO_HONG_VASCONCELOS_2017} \textbf{$S_{2}$-fication of  Buchsbaum local rings:}   Let $A$ be a Noetherian ring with the total ring of fractions $\mathbb K$. The smallest finite extension $\phi: A\longrightarrow \mathbb S \subseteq \mathbb K$ satisfying Serre's condition $(S_{2})$ is called the $S_{2}$-fication of $A$. Let $(A,\m)$ be a Buchsbaum local ring with positive depth.  Then from \cite[Theorem 4.2]{GHEZZI_GOTO_HONG_VASCONCELOS_2017}, $\mathbb S=\Hom_{A}(\m,\m)$ is the $S_{2}$-fication of $A.$ From \cite[Proposition 4.1(2)]{GHEZZI_GOTO_HONG_VASCONCELOS_2017},  $\mathbb S$ is semi-local, therefore, let $\{\mathfrak M_{1},\ldots, \mathfrak M_{l}\}$ be the maximal ideals of $\mathbb S,$ and let $\mathbb S_{1},\ldots,\mathbb S_{l}$ denote the corresponding localizations. Further, for each $1\leq i\leq l,$ let $f_{i}$ be the relative degree $[\mathbb S/\mathfrak M_{i}:A/\mathfrak m].$
\end{point}

\begin{proposition} \label{buchsbaumring}
    Let $(A,\m)$ be a Buchsbaum local ring of dimension $d \geq 2$ with $\depth A\geq d-1$, then $e_{2}(\m)\geq e_{1}(\m)-e_{0}(\m)+1.$
\end{proposition}

\begin{proof}
  We prove by induction on the dimension $d.$ For $d=2$,  let $\mathbb{S}$ be the $S_{2}$-fication of $A.$  Using \cite[Proposition 4.1]{GHEZZI_GOTO_HONG_VASCONCELOS_2017}  and  from \cite[Chapter 2, Proposition 2.1$(iii)$]{stückrad2014buchsbaum}, $\mathfrak{m} H
    ^{1}_{\mathfrak{m}}(A) = 0$, we have that $\mathfrak m =\mathfrak m\mathbb S.$ Thus, from \cite[Lemma 4.4]{mandal2024}, $e_{i}(\m)=e_{i}^{A}(\m\mathbb S)$ for $i=0,1$ and $e_{2}(\m)=e_{2}^{A}(\m \mathbb S)+e_{1}(J),$ where $J$ is any minimal reduction of $\m.$ Therefore,
    \begin{align}
        \nonumber
        e_{2}(\m)&=e_{2}^{A}(\m \mathbb S)+e_{1}(J)\\ \nonumber
        &=\sum_{i=1}^{l}e_{2}^{\mathbb S_{i}}(\m \mathbb S_{i})f_{i}+e_{1}(J) & \text{ (from \cite[Lemma 4.2$(iii)$]{mandal2024})}\\ \label{in6}
        &\geq \sum_{i=1}^{l}\left[e_{1}^{\mathbb S_{i}}(\m \mathbb S_{i})-e_{0}^{\mathbb S_{i}}(\m \mathbb S_{i})+\lambda_{\mathbb S_{i}}(\mathbb S_{i}/\m \mathbb S_{i})\right]f_{i}+e_{1}(J)\\ \nonumber
        &=\sum_{i=1}^{l}e_{1}^{\mathbb S_{i}}(\m \mathbb S_{i})f_{i}-\sum_{i=1}^{l}e_{0}^{\mathbb S_{i}}(\m \mathbb S_{i})f_{i}+\sum_{i=1}^{l}\lambda_{\mathbb S_{i}}(\mathbb S_{i}/\m \mathbb S_{i})f_{i}+e_{1}(J)\\ \nonumber
        &=e_{1}^{A}(\m \mathbb S)-e_{0}^{A}(\m \mathbb S)+\lambda_{A}(\mathbb S/\m \mathbb S)+e_{1}(J) & \text{ (from \cite[Lemma 4.2$(i)$ and $(ii)$]{mandal2024})}\\ \nonumber
        &=e_{1}(\m)-e_{0}(\m)+\lambda(A/\m )+\lambda_{A}(\mathbb S/A)+e_{1}(J)& \text{ (since $\m=\m \mathbb S $)}\\  \label{in7}
        &=e_{1}(\m)-e_{0}(\m)+1.
    \end{align}

     Note that the inequality in (\ref{in6})  holds since  $\mathbb S_{i}$ is Cohen- Macaulay of dimension 2 for all $1 \leq i \leq l,$ thus,  $e_{2}^{\mathbb S_{i}}(\m \mathbb S_{i})\geq e_{1}^{\mathbb S_{i}}(\m \mathbb S_{i})-e_{0}^{\mathbb S_{i}}(\m \mathbb S_{i})+1$ for all $1 \leq i \leq l$ from \cite{itoh1}. Furthermore, (\ref{in7}) holds since  $\lambda_{A}(\mathbb S/A)=\lambda\left(H^{1}_{\mathfrak m}(A)\right)=-e_{1}(J)$ because $A$ is Buchsbaum \cite[
         Chapter 1, Propositions 2.6 and 2.7]{stückrad2014buchsbaum}.

      Assume $d \geq 3$ and that our assertion holds for $d-1.$  Let $x$ be a superficial element for $\m.$ Set $A_{1}=A/(x)$ and $\m_{1}=\m A_{1}.$ Then $A_{1}$ is a Buchsbaum ring of dimension $d-1$ with $\depth A_{1} \geq d-2.$ Thus, from induction hypothesis, we have $e_{2}(\m_{1})\geq e_{1}(\m_{1})-e_{0}(\m_{1})+1.$ Further, from \cite[Proposition 1.2]{rossi2010hilbert}, we have $e_{i}(\m)=e_{i}(\m/(x))$ for all $i=0,1,2$. Therefore, $e_{2}(\m)\geq e_{1}(\m)-e_{0}(\m)+1.$
\end{proof}
 As an immediate consequence, we have the following corollary for two-dimensional Buchsbaum local rings with depth zero.  
 
 \begin{corollary} \label{corobuch}
     Let $(A,\m)$ be a two-dimensional Buchsbaum ring with depth zero, then $e_{2}(\m)\geq e_{1}(\m)-e_{0}(\m)+1+e_{2}(J),$ where $J$ is any parameter ideal in $A.$
 \end{corollary}

 \begin{proof}
     Let  $J$ be is any parameter ideal in $A$.
     Set $\overline{A}=A/\left(H_{\mathfrak m}^{0}(A)\right)$  and   $ \overline{\m}=\m\overline{A}.$   Note that $\overline{A}$ is a two-dimensional Buchsbaum ring with positive depth. From \cite[Proposition 2.3]{rossi2010hilbert},  $e_{i}(\overline{\m})=e_{i}(\m)$ for $i=0,1.$ Therefore, from \cite[Proposition 2.3]{rossi2010hilbert}, we have
      \begin{align*}
       e_{2}(\m)&=e_{2}(\overline{\m})+\lambda(H_{\m}^{0}(A))\\
       &\geq e_{1}(\overline{\m})-e_{0}(\overline{\m})+1+\lambda(H_{\m}^{0}(A)) & \text{ (from Proposition \ref{buchsbaumring})}\\
       & = e_{1}(\m)-e_{0}(\m)+1+e_{2}(J) & \text{ (from \cite[Corollary 4.2]{MR684272})}.& \qedhere
      \end{align*}
 \end{proof} 
 The following examples show that the lower bounds in Proposition \ref{buchsbaumring} and Corollary \ref{corobuch}, respectively, are sharp.   
 \begin{example} \normalfont
 \begin{enumerate} [(i)]
     \item \cite[Example 4.7]{GHEZZI_GOTO_HONG_VASCONCELOS_2017}
     Let $A=\mathbb Q[[X,Y,Z,W]]/(X,Y)\cap(Z,W),$ a two-dimensional Buchsbaum ring with positive depth, and let $\m$ denote the maximal ideal of $A.$ Note that the $S_{2}$-fication of $A$ is $\mathbb{S}=\mathbb Q[[z,w]]\oplus \mathbb Q[[x,y]].$  For every $n \geq 0$, we have the following short exact sequence:
\[0\longrightarrow A/\m^{n+1}\longrightarrow  \mathbb Q[[x,y,z,w]]/[(x,y)+\m^{n+1}] \oplus   \mathbb Q[[x,y,z,w]]/[(z,w)+\m^{n+1}]\longrightarrow \mathbb Q \longrightarrow 0.\]
Then for all $n \geq 0, $ we have
\begin{equation*}
    \begin{split}
    \lambda (A/\m^{n+1})&=\lambda\left( \mathbb Q[[x,y,z,w]]/[(x,y)+\m^{n+1}] \right)+\lambda\left( \mathbb Q[[x,y,z,w]]/[(z,w)+\m^{n+1}] \right)-1\\
    &=2\binom{n+2}{2}-1.
\end{split}
\end{equation*}
     Therefore, $e_{0}(\m)=2$, $e_{1}(\m)=0$ and $e_{2}(\m)=-1.$ In particular, one has $e_{2}(\m)=e_{1}(\m)-e_{0}(\m)+1.$

     \item \cite[Remark 2.8]{goto1} Let $A=\mathbb Q[[X,Y,Z]]/(XZ^2,YZ^2,Z^3),$   a two-dimensional Buchsbaum ring with depth zero. Let $\m$ denote the maximal ideal of $A.$ By CoCoA \cite{cocoa}, the Hilbert series of $\m$ is
     \begin{equation*}
         HS_{\m}(t)=\frac{1 + t + t^2 - 2t^3 + t^4}{(1-t)^2}.
     \end{equation*}
     Here $e_{0}(\m)=2$, $e_{1}(\m)=1$, $e_{2}(\m)=1$ and $\lambda\left(H^{0}_{\m}(A)\right)=1=e_{2}(J),$ for any parameter ideal $J$ in $A$. In particular, one has $e_{2}(\m)=e_{1}(\m)-e_{0}(\m)+1+e_{2}(J).$
 \end{enumerate}
 \end{example}

\section[]{ $e_{2}(I)= e_{1}(I)-e_{0}(I)+\lambda(A/I)+1$}

In this section, we study the consequences of the numerical condition $e_{2}(I)=e_{1}(I)-e_{0}(I)+\lambda(A/I)+1$ for $I$ integrally closed. We mostly restrict our attention to  two-dimensional Cohen-Macaulay rings.
We show that if $d \leq 1,$ then this  condition forces the $h$-polynomial of $I$ to have degree exactly 3. We also prove that for $d\geq 2,$ the equality $e_{2}(I)=e_{1}(I)-e_{0}(I)+\lambda(A/I)+1$ implies that $e_{3}(I) \leq 1.$ Moreover, if equality holds, then the Ratliff-Rush filtration with respect to $I$ behaves modulo a superficial sequence of length $d-1.$ 
We begin with the following proposition. 

\begin{proposition} \label{csk}
    Let $(A, \m) $ be a Cohen-Macaulay local ring of dimension less than or equal to $1$ and $I$ an  $\m$-primary integrally closed ideal. If $e_{2}(I)=e_{1}(I)-e_{0}(I)+\lambda(A/I)+1$ then the following hold:
    \begin{enumerate}[\normalfont(i)]
        \item \label{csk1} $\deg h_I(t)= 3.$
        \item \label{csk2} For $d=1$ and $J=(x)$ a minimal reduction ideal of $I$, if $I^3\nsubseteq J,$ then $G(I)$ is Cohen-Macaulay.
    \end{enumerate}
    
\end{proposition}

\begin{proof}
  (\ref{csk1}) For $d=0$, from the proof of Lemma \ref{prop4.1}, it follows that $\rho_{3}(I)+3\rho_{4}(I)+\ldots+\dbinom{s-1}{2} \rho_{s}(I)=1.$ This implies $\rho_{j}(I)=0$ for all $j \geq 4$ and $\rho_3(I)=1.$ Hence,  $\deg h_I(t) =3.$ 

For $d=1$, from the proof of Lemma \ref{prop4.1}, it follows that $\displaystyle \sum_{j\geq2}(j-1)\nu_{j}(I)=1$ for all $j\geq 0,$ where $\nu_j(I)=\lambda(I^{j+1}/JI^j)$. This implies that $\nu_{j}(I)=0$ for all $j \geq 3$ and $\nu_2(I)=1.$ From \cite[Theorem 2.5$(c)$]{rossi2010hilbert}, we get $\deg h_I(t) = 3,$ and therefore,  $h_{I}(t)=\nu_{0}(I)+(e_{0}(I)-\nu_{1}(I)-\lambda(A/I))t+(\nu_{1}(I)-1))t^{2}+t^3.$

\noindent
(\ref{csk2}) From part (\ref{csk1}) and the proof of Lemma \ref{prop4.1}, we have $s^*(I)\leq r_J(I)=3.$ Since $I$ is integrally closed, thus $\wt{I}=I$. Therefore, it suffices to show that $\wt{I^2}=I^2.$ 
From \cite[1.5(b)]{Part2}, we have $e_{0}(I)=\widetilde{e}_{0}(I)$,   $e_{1}(I)=\widetilde{e}_{1}(I)$ and $e_{2}(I)=\widetilde{e}_{2}(I)+\displaystyle \sum_{n \geq 0}\lambda(\widetilde{I^{n+1}}/I^{n+1})$. Thus, $e_{2}(I)=\widetilde{e}_{2}(I)+\lambda(\widetilde{I^{2}}/I^{2}).$ Hence,
 \begin{equation}\label{e2tilde}
     \widetilde{e}_{2}(I)+\lambda(\widetilde{I^{2}}/I^{2})=\widetilde{e}_{1}(I)-\widetilde{e}_{0}(I)+\lambda(A/\wt{I})+1.
 \end{equation}
  From \cite[Theorem 2.5$(c)$]{rossi2010hilbert}, we have $\wt{e}_0(I)=\lambda(A/J),$ $\wt{e}_{1}(I)=\displaystyle \sum_{j\geq0}\wt{\nu_j}(I),$ and $\wt{e}_{2}(I)=\displaystyle \sum_{j\geq1}j\wt{\nu_j}(I),$ where $\wt{\nu_j}(I)=\lambda(\wt{I^{j+1}}/J\wt{I^j})$ for all $j\geq 0.$ Therefore, from \eqref{e2tilde}, we obtain $\lambda(\wt{I^2}/I^2)+\lambda(\wt{I^3}/J\wt{I^2})=1.$ If $\lambda(\wt{I^2}/I^2)=1$ then $I^3\subseteq \wt{I^3}=J\wt{I^2}\subseteq J,$ which is a contradiction to $I^3\nsubseteq J.$ Thus, $\widetilde{I^{2}}=I^{2}$.
\end{proof}

We now consider the second extremal case 
for integrally closed 
$\m$-primary ideals in dimension two. The following proposition describes the depth of 
$G(I)$ and gives a criterion for the Cohen–Macaulayness of 
$\wt{G}(I).$

\begin{proposition}
\label{depth}
    Let $(A, \m) $ be a two dimensional Cohen-Macaulay local ring and $I$ an $\m$-primary integrally closed ideal with minimal reduction $J$. If $e_{2}(I)=e_{1}(I)-e_{0}(I)+\lambda(A/I)+1$ and $I^3\nsubseteq J$ then the following hold:
    \begin{enumerate}[\normalfont(i)]
        \item \label{depth1}  $\depth G(I)=0$ or $2$,
        \item  \label{depth2} $\wt{G}(I)$ is Cohen-Macaulay if and only if $\wt{I^2}\cap J=JI$. In particular, $\wt{G}(\m)$ is Cohen-Macaulay.
    \end{enumerate}
\end{proposition}

\begin{proof}
(\ref{depth1})
Since the residue field is infinite, we choose $x_1 \in I\setminus I^2$, such that $I/(x_1)$ is integrally closed. Set $A_1=A/(x_1)$, $I_1=IA_1$ and  $J_1=JA_1.$   If $\depth G(I)=0,$ there is nothing to prove. Hence, assume that  $\operatorname{depth}G(I) \neq 0.$ Since $\lambda(A_1/I_1)=\lambda(A/I),$ and  by \cite[Proposition 1.2]{rossi2010hilbert}, we have $e_{i}(I_1)=e_{i}(I)$ for $0\leq i\leq 2.$ Thus, $e_{2}(I_1)=e_{1}(I_1)-e_{0}(I_1)+\lambda(A_1/I_1)+1$. Since $I^3\nsubseteq J$, we have $I^3_1\nsubseteq J_{1}$. Therefore, by Proposition \ref{csk} (\ref{csk2}), it follows that $\operatorname{depth}G(I_1)=1.$ Hence, by Sally's descent $\operatorname{depth}G(I)=2.$

\noindent (\ref{depth2})
  By part (\ref{depth1}), we have $\depth G(I)=0 \text{ or }2.$ If $\depth G(I)=2,$ then we have nothing to prove. Assume that $\depth G(I)=0.$ Set $\wt{v_j}(I)=\lambda(\wt{I^{j+1}}/J\wt{I^j})$ for all $j\geq 0.$ From \cite[Equation 3.1]{rossi2010hilbert}, we have $e_1(I)= \displaystyle \sum_{j\geq 0}\wt{v_j}(I)$ and $e_2(I)=\displaystyle \sum_{j\geq 1}j\wt{v_j}(I).$ 
  Since $e_{2}(I)=e_{1}(I)-e_{0}(I)+\lambda(A/I)+1$, we observe that $v_2(I)=1$, $v_j(I)=0$ for all $j\geq 3.$ It follows that  $\wt{I^{n+1}}\cap J=J\wt{I^n}$ for all $n\geq 3$. We have the following short exact sequence: $$0\longrightarrow \wt{I^3}\cap J/J\wt{I^2}\longrightarrow \wt{I^3}/J\wt{I^2}\longrightarrow \wt{I^3}/\wt{I^3}\cap J\to 0.$$ Since $I^3 \nsubseteq  J$, we have $\wt{I^3}\cap J=J\wt{I^2}$.
Indeed, if $\widetilde{I^3} = \widetilde{I^3}\cap J$, then $I^3 \subseteq \widetilde{I^3} \subseteq J$,
which is a contradiction. Thus the result now follows from  \cite[Theorem 1.1]{rossi2010hilbert}.
Moreover, since $\widetilde{\m^2}\cap J = J\m$ (see \cite[2.10]{ankit}), it follows that $\widetilde{G}(\m)$ is Cohen-Macaulay.
   \end{proof}

We now prove the main result of this section. In the following theorem, we establish an upper bound for $e_{3}(I)$ whenever $e_2(I)=e_1(I)-e_0(I)+\lambda(A/I)+1$.

\begin{theorem}\label{e3=1}
Let $(A,\m) $ be a Cohen-Macaulay local ring of dimension $d \geq 2$ and let $I$ be an $\m$-primary integrally closed ideal. Let $J$ be a minimal reduction of $I$ with $I^3\nsubseteq J.$ If $e_2(I)=e_1(I)-e_0(I)+\lambda(A/I)+1$ and $\wt{I^2}\cap J=JI$, then the following hold:
\begin{enumerate} [\normalfont(i)]
    \item  \label{e3=11}For $d=2,$ $e_3(I)\leq 1.$ Further, if  $e_3(I)=1$ then $G(I)$ is Cohen-Macaulay.

    \item\label{e3=12} For $d\geq 3,$ $e_3(I)\leq 1.$ Further, if $d=3$ and $e_3(I)=1$, then $\wt{G}(I)$ is Cohen-Macaulay.  
\end{enumerate}
\end{theorem}
\begin{proof} 
(\ref{e3=11}) Set $\wt{\nu_j}(I)=\lambda(\wt{I^{j+1}}/J\wt{I^j})$ for all $j\geq 0.$ By \cite[Theorem 2.5$(c)$]{rossi2010hilbert} and the proof of Proposition \ref{depth}, we get $ \wt{e}_{3}(I)= \displaystyle \sum _{j \geq 2}\binom{j}{2}\wt{\nu_j}(I)=1.$ Furthermore, from the proof of Theorem \ref{e3bd} (\ref{e3bd(i)}), we have $ e_{3}(I)=1- \displaystyle\sum_{j\geq 0}\lambda\left(\frac{\widetilde{I^{j+1}}}{I^{j+1}}\right)\leq 1.$
    Suppose equality holds, then $\widetilde{I^{j}}=I^{j}$ for all $j \geq 1.$ This implies $\depth G(I) >0.$ The result now follows from Proposition \ref{depth}.

\noindent
(\ref{e3=12})   We proceed by induction on the dimension $d$. Suppose $d=3.$ We choose $x \in I\setminus I^2$ a superficial element for $I,$ such that $I/(x)$ is integrally closed. From the proof of Theorem \ref{e3bd} (\ref{e3bd(ii)}), we get $e_3(I)\leq 1$. Suppose equality holds, then the Ratliff-Rush filtration of $I$ behaves well modulo $(x).$ Therefore, it follows that  $\depth \widetilde{G}(I) \geq 2.$ Furthermore, from \cite[Remark 4.5]{Part2}, $\widetilde{G}(I)/x\widetilde{G}(I)=\widetilde{G}({I_1}).$ We also have $\wt{I^2_1}\cap J_1=J_1I_1$. Therefore, from part (\ref{e3=11}), $\wt{G}(I_1)$ is Cohen-Macaulay. Thus from Sally's descent $\widetilde{G}(I)$ is Cohen-Macaulay.

Suppose $d\geq 4.$ Let $x \in I$ be chosen as in the case $d=3.$ By an argument similar to that used in the proof of Theorem \ref{e3bd} (\ref{e3bd(ii)}), we obtain  $e_3(I)\leq 1.$ 
\end{proof}

The following example provides an infinite class of examples where $e_{3}(\m)=1$.

\begin{example} \normalfont
   Let $A=k[|x,y,z_1,\ldots z_d|]/(x^3,xy^2,y^3)$ and $I=\m$. Then $\dim(A)=d$. Set $J=(z_1,\ldots,z_d)$. It is easy to verify that $J$ is a minimal reduction of $\mathfrak{m}$, and that $\mathfrak{m}^3 \nsubseteq J$.
The associated graded ring $G(\mathfrak{m})$ is Cohen-Macaulay. The Hilbert series of $\m$ is:  $$HS_{\m}(t)=\frac{1+2t+3t^2+t^3}{(1-t)^{d}}.$$ Hence, $e_0(\m)=7, e_1(\m)=11, e_2(\m)=6 \text{ and } e_3(\m)=1.$  In particular, $e_2(\m)=e_1(\m)-e_0(\m)+\lambda(A/\m)+1,$ and hence the hypothesis of Theorem \ref{e3=1} is satisfied. When $d=2$, this gives  an example of Proposition \ref{depth}, where the depth of $G(\m)$ is $2.$ 
\end{example}

We conclude this section with the following question.
\begin{question} \normalfont
Let $(A, \mathfrak m)$ be a Cohen-Macaulay local ring of dimension $d$, and let $I$ be an $\mathfrak m$-primary integrally closed ideal satisfying $e_2(I)=e_1(I)-e_0(I)+\lambda(A/I)+1.$ Is it possible to obtain a lower bound for $\type_I(A)$, analogous to Theorem \ref{bd}? Moreover, does equality  force   $G(I)$ to be Cohen-Macaulay?

\end{question}

\providecommand{\bysame}{\leavevmode\hbox to3em{\hrulefill}\thinspace}
\providecommand{\MR}{\relax\ifhmode\unskip\space\fi MR }
\providecommand{\MRhref}[2]{
  \href{http://www.ams.org/mathscinet-getitem?mr=#1}{#2}
}

\end{document}